\documentclass[a4paper,11pt]{amsart}
\usepackage{amssymb}
\usepackage{amsmath}
\usepackage{amsthm}
\usepackage{amscd}
\theoremstyle{plain}
\newtheorem{thm}{Theorem}[section]
\newtheorem{prop}[thm]{Proposition}

\newtheorem{cor}[thm]{Corollary}
\newtheorem{defn}[thm]{Definition}

\newtheorem{rem}[thm]{Remark}

\def\vol{\mathop{\mathrm{Vol}}\nolimits}

\numberwithin{equation}{section}

\title{Algebro-geometric semistability of polarized toric
manifolds}
\author{Hajime Ono}
\address{Department of Mathematics,
Faculty of Science and Technology,
Tokyo University of Science,
2641 Yamazaki, Noda,
Chiba 278-8510, Japan}
\email{ono\_hajime@ma.noda.tus.ac.jp}
\date{}

\begin{document}

\maketitle

\begin{abstract}
Let $\Delta\subset \mathbb{R}^n$ be an
$n$-dimensional integral Delzant polytope.
It is well-known that there exist the $n$-dimensional compact toric
manifold $X_\Delta$ and the very ample
$(\mathbb{C}^\times)^n$-equivariant line bundle $L_\Delta$ on $X_\Delta$
associated with $\Delta$.
In the present paper, we give a necessary and sufficient condition
for Chow semistability of $(X_\Delta,L_\Delta^i)$ for a maximal torus
action. We then see that asymptotic (relative) Chow semistability
implies (relative)
K-semistability for toric degenerations,
which is proved by Ross and Thomas \cite{rt},
without any knowledge of Riemann-Roch theorem and test configurations.
\end{abstract}

\section{Introduction}

Let $X$ be a compact complex manifold and $L$ an ample line bundle
on $X$. We call the pair $(X,L)$ a polarized manifold.
The one of the main subjects in K\"ahler geometry is the existence
problem of K\"ahler metrics with constant scalar curvature, more general,
of extremal K\"ahler metrics.
It is now conjectured that the existence of such metrics
in $c_1(L)$ is equivalent to some algebro-geometric stability
of $(X,L)$.
The one of the difficulty to consider this problem is the existence of
many notions of stability. We have to specify the exact one.
Though a lot of progress
is made recently in this problem by Tian, Donaldson and
many other researchers, we do not know what is appropriate stability yet.
Therefore it is important to know the relation among various notions
of stability.

In the present paper we investigate the following semistabilities
of polarized toric manifolds, asymptotic Chow semistability,
K-semistability and relative K-semistability.

Let $\Delta\subset \mathbb{R}^n$ be an $n$-dimensional integral
Delzant polytope. Namely, $\Delta$ satisfies the following conditions
(in \cite{o} polytopes satisfying these conditions are called
absolutely simple)
:
\begin{enumerate}
	\item The vertices ${\bf w}_1,\dots,{\bf w}_d$ of $\Delta$ are contained
	in $\mathbb{Z}^n$.
	\item For each $l$, there are just $n$
	rational edges $e_{l,1},
	\dots,e_{l,n}$ of $\Delta$ emanating from ${\bf w}_l$.
	\item The primitive vectors with respect to the edges $e_{l,1},
	\dots,e_{l,n}$ generate the lattice $\mathbb{Z}^n$ over $\mathbb{Z}$.
\end{enumerate}

It is well-known that $n$-dimensional integral Delzant polytopes
correspond to $n$-dimensional compact toric manifolds with 
$(\mathbb{C}^\times)^n$-equivariant very ample line bundles.
The reader is referred to \cite{o} for example.

The following is the main theorem of this paper.
(We refer the reader to the subsequent sections for notations.)

\begin{thm}\label{thm:1.1}
Let $\Delta\subset \mathbb{R}^n$ be an $n$-dimensional integral
Delzant polytope. For a positive integer $i$, 
the Chow form of
$(X_\Delta,L_\Delta^i)$
is $T^{E_\Delta(i)-1}_{\mathbb{C}}$-semistable if and only if
\begin{equation}\label{e:1.1}
P_\Delta(i;g):=E_\Delta(i)\int_{\Delta} g dv
-
\vol(\Delta)\sum_{{\bf a}\in \Delta\cap (\mathbb{Z}/i)^n}
g({\bf a})
\ge 0
\end{equation}
holds for any concave piecewise linear function $g:\Delta\to \mathbb{R}$
contained in
$PL(\Delta;i)$.
\end{thm}


This paper is organized as follows. In Section $2$ we fix some notations.
In Section $3$ we review some basics of geometric invariant theory
briefly for later use. In Section $4$, we first define the Chow form
of submanifolds of a projective space.
Then Chow semistability is defined as
GIT-semistability of the Chow form. Finally we prove Theorem \ref{thm:1.1}.
In Section $5$, we compare asymptotic Chow semistability
with K-semistability of polarized toric manifolds
through Theorem \ref{thm:1.1}.
In Section $6$, by analogy with Chow semistability,
we define the notion of relative Chow semistability of polarized toric
manifolds in toric sense and prove that asymptotic relative Chow semistability
in toric sense
implies relative K-semistability for toric degenerations.

\section{Preliminaries}

In this section we fix some notations. Let 
$\Delta\subset \mathbb{R}^n$ be an $n$-dimensional integral Delzant
polytope.
We denote the Ehrhart polynomial
of $\Delta$ by
$E_\Delta(t)$, which is a polynomial of degree $n$ satisfying
$$
E_\Delta(i)=\#(i\Delta\cap \mathbb{Z}^n)=\#(\Delta\cap (\mathbb{Z}/i)^n)
$$
for each positive integer $i$. It is well-known that
such a polynomial exists and the coefficients of 
$n$-th and $(n-1)$-th order terms are
$\vol(\Delta)$ and $\vol(\partial \Delta)/2$ respectively;
\begin{equation}\label{e:2.1}
E_\Delta(t)= \vol(\Delta)t^n+\frac{\vol(\partial \Delta)}{2}t^{n-1}
+O(t^{n-2}).
\end{equation}
Here the volume form $d\sigma$ on $\partial \Delta$ is defined as follows.
On a facet $\{h_r=c_r\}\cap \Delta$ of $\Delta$, where $h_r$ is a
primitive linear form, $dh_r\wedge d\sigma$ equals to the Euclidean
volume form $dv$.

In Section $4$ we will consider representations of the complex torus
$(\mathbb{C}^\times)^{E_\Delta(i)}$. The character group of this torus
can be identified with
$\{i\Delta\cap \mathbb{Z}^n\to \mathbb{Z}\}\simeq \{\Delta\cap (\mathbb{Z}/i)^n
\to \mathbb{Z}\}\simeq \mathbb{Z}^{E_\Delta(i)}$.
Then we denote
\begin{equation*}
W(i\Delta):=\{i\Delta\cap \mathbb{Z}^n\to \mathbb{R}\}
\simeq \{\Delta\cap (\mathbb{Z}/i)^n
\to \mathbb{R}\}\simeq \mathbb{R}^{E_\Delta(i)},
\end{equation*}
\begin{equation*}
W(i\Delta)_{\mathbb{Q}}:=\{i\Delta\cap \mathbb{Z}^n\to \mathbb{Q}\}
\simeq \{\Delta\cap (\mathbb{Z}/i)^n
\to \mathbb{Q}\}\simeq \mathbb{Q}^{E_\Delta(i)}.
\end{equation*}
We identify $W(i\Delta)$ with its dual space by the scalar product
\begin{equation*}
(\varphi,\psi):=\sum_{{\bf a}\in \Delta\cap (\mathbb{Z}/i)^n}\varphi({\bf a})
\psi({\bf a}).
\end{equation*}

For each $\varphi\in W(i\Delta)$, define a
concave piecewise linear function $g_\varphi:\Delta\to \mathbb{R}$ as follows.
Let
\begin{equation*}
G_\varphi:=\text{the convex hull of }
\bigcup_{{\bf a} \in \Delta\cap (\mathbb{Z}/i)^n}\{({\bf a},t)\,|\,
t\le \varphi({\bf a})\}
\subset \mathbb{R}^n \times \mathbb{R}.
\end{equation*}
Then we define $g_\varphi$ by
$$g_\varphi({\bf x}):=\max\{t\,|\,
({\bf x},t)\in G_\varphi\},\ \ {\bf x}\in \Delta$$
and denote
$$PL(\Delta;i):=\{g_\varphi\,|\,
\varphi\in W(i\Delta)\},\ \ 
PL_\mathbb{Q}(\Delta;i):=\{g_\varphi\,|\,
\varphi\in W_{\mathbb{Q}}(i\Delta)\}.
$$
\begin{rem}
It is easy to see that
if $g:\Delta\to \mathbb{R}$ is a rational concave piecewise linear
function, then there is a positive integer $i$ such that
$g\in PL_{\mathbb{Q}}(\Delta;i)$.
\end{rem}
\begin{rem}
If $g\in PL(\Delta;i)$ (resp. $g\in PL_{\mathbb{Q}}(\Delta;i)$),
then $g\in PL(\Delta;ki)$ (resp. $g\in PL_{\mathbb{Q}}(\Delta;ki)$)
for any positive integer $k$.
\end{rem}

\section{GIT-stability}

Let $G$ be a reductive Lie group. Suppose that $G$ acts on a
complex vector space $V$ linearly. We call a nonzero vector $v\in V$
$G$-semistable
if the closure of the orbit $Gv$ does not contain the origin.
Similarly we call $p\in \mathbb{P}(V)$ $G$-semistable if any representative
of $p$ in $V\setminus\{{\bf 0}\}$ is $G$-semistable.
It is well-known that there is the following
good criterion for $v$ being $G$-semistable, see
\cite{git}.

\begin{prop}[Hilbert-Mumford criterion, \cite{git}]\label{prop:3.1}
$p\in \mathbb{P}(V)$ is $G$-semistable if and only if
$p$ is $H$-semistable for each maximal torus $H\subset G$.
\end{prop}


Hence it is important to study $G$-semistability
when $G$ is isomorphic to an algebraic torus $(\mathbb{C}^\times )^n$.
Let $G$ be isomorphic to $(\mathbb{C}^\times )^n$.
Then a $G$-module $V$ is decomposed as
\begin{equation}\label{e:3.1}
V=\sum_{\chi\in \chi(G)}V_\chi,\ \ V_\chi:=\{v\in V\,|\,
t\cdot v=\chi(t)v,\ \forall t\in G\},
\end{equation}
where $\chi(G)\simeq \mathbb{Z}^n$ is the character group of the torus $G$.

\begin{defn}\label{def:3.2}
Let $v=\sum_{\chi\in \chi(G)}v_\chi$ be a nonzero vector in $V$.
The weight polytope $\textup{Wt}\,_G(v)\subset \chi(G)\otimes_{\mathbb{Z}}
\mathbb{R}$ of $v$ is the convex hull of $\{\chi\in \chi(G)\,|\,
v_\chi\not={\bf 0}\}$ in $\chi(G)\otimes_{\mathbb{Z}}
\mathbb{R}$.
\end{defn}

The following fact about $G$-semistability is standard.

\begin{prop}\label{prop:3.3}
Let $G$ be isomorphic to $(\mathbb{C}^\times)^n$. Suppose that $G$ acts
a complex vector space $V$ linearly. Then a nonzero vector $v\in V$ is
$G$-semistable if and only if the weight polytope $\textup{Wt}\,_G(v)$
contains the origin.
\end{prop}

Let $G=(\mathbb{C}^\times)^n$ and $H$ be the subtorus
\begin{equation}\label{e:3.2}
H=\{(t_1,\dots,t_{n-1},(t_1\cdots t_{n-1})^{-1})\,|\,
(t_1,\dots,t_{n-1})\in (\mathbb{C}^\times)^{n-1}
\}\simeq (\mathbb{C}^\times)^{n-1}.
\end{equation}
Then the weight polytope $\text{Wt}\,_H(v)\subset \chi(H)\otimes_{\mathbb{Z}}
\mathbb{R}\simeq \mathbb{R}^{n-1}$ equals to $\pi(\text{Wt}\,_G(v))$, where
the linear map $\pi:\mathbb{R}^n\to \mathbb{R}^{n-1}$ is given as
$(x_1,\dots,x_{n-1},x_n)\mapsto (x_1-x_n,\dots,x_{n-1}-x_n)$.

Therefore we see the following.

\begin{prop}\label{prop:3.4}
$v$ is $H$-semistable if and only if there exists $t\in \mathbb{R}$
such that $(t,\dots,t)\in \textup{Wt}\,_G(v)$.
\end{prop}

\section{Chow semistability of polarized toric manifolds}

We first define the Chow form of irreducible projective varieties.
See \cite{gkz} for more detail. 

\begin{defn}\label{def:4.1}
Let $X\subset \mathbb{C}P^N$ be an $n$-dimensional irreducible subvariety of
degree $d$. It is easy to see that the subset $Z_X$ of the
Grassmannian $\textup{Gr}(N-n-1,\mathbb{C}P^N)$ defined by
$$Z_X=\{L\in \textup{Gr}(N-n-1,\mathbb{C}P^N)\,|\,
L\cap X\not=\emptyset\}$$
is an irreducible hypersurface of degree $d$.
Hence $Z_X$ is given by the vanishing of a degree $d$ element $[R_X]\in
\mathbb{P}(\mathcal B_d(N-n-1,\mathbb{C}P^N))$, where
$\mathcal B(N-n-1,\mathbb{C}P^N)=\oplus_{d}\mathcal B_d(N-n-1,\mathbb{C}P^N)$
is the graded coordinate ring of the Grassmannian. We call $R_X$ the
Chow form of $X$.
\end{defn}

Since the special linear group $SL(N+1,\mathbb{C})$ acts naturally on
$\mathcal B_d(N-n-1,\mathbb{C}P^N)$, we can consider the 
$SL(N+1,\mathbb{C})$-stability
of the Chow form $R_X$.
\begin{defn}\label{def:4.2}
Let $X\subset \mathbb{C}P^N$ be an $n$-dimensional irreducible subvariety
of degree $d$. We call $X$ Chow semistable if the Chow form $R_X$
is $SL(N+1,\mathbb{C})$-semistable.
When $X$ is not Chow semistable $X$ is called Chow unstable.
\end{defn}

\begin{defn}\label{def:4.3}
Let $L$ be a very ample line bundle on a complex manifold $X$.
$(X,L)$ is called 
Chow semistable when $\Psi(X)\subset \mathbb{P}(H^0(X;L)^*)$
is Chow semistable.
Here $\Psi:X\to 
\mathbb{P}(H^0(X;L)^*)$
is the Kodaira embedding.
\end{defn}

\begin{defn}\label{def:4.4}
Let $(X,L)$ be a polarized manifold. We call $(X,L)$ asymptotically
Chow semistable if $(X,L^i)$ is Chow semistable for any sufficiently 
large integer $i$.
\end{defn}

Now we investigate Chow semistability of $(X_\Delta,L_\Delta^i)$
for an $n$-dimensional integral Delzant polytope $\Delta$ and a positive
integer $i$. By Hilbert-Mumford criterion, Proposition \ref{prop:3.1},
$H$-semistability of the Chow form is essential for any maximal torus
$H$ of $SL(E_\Delta(i))$.
As a specific case, we take the following maximal torus;
$$
T_{i\Delta}^{\mathbb{C}}:=(\mathbb{C}^\times )^{E_\Delta(i)}\cap
SL(E_\Delta(i)).
$$
Here $(\mathbb{C}^\times)^{E_\Delta(i)}\subset GL(E_\Delta(i))$
is the torus consisting of diagonal matrices. Note that 
$T_{i\Delta}^{\mathbb{C}}$ is the subtorus of
$(\mathbb{C}^\times)^{E_\Delta(i)}$ given as \eqref{e:3.2}.

\begin{prop}[\cite{gkz}, see also \cite{ono}]\label{prop:4.5}
Let $Ch_{i\Delta}$ be the weight polytope of the Chow form of
$(X_\Delta,L_\Delta^i)$ for $(\mathbb{C}^\times)^{E_\Delta(i)}$-action.
Then the affine hull of $Ch_{i\Delta}$ in $W(i\Delta)$ is
\begin{equation}\label{e:4.1}
\left\{\varphi\in W(i\Delta)\,\Bigg|\,
\sum_{{\bf a}\in \Delta\cap (\mathbb{Z}/i)^n}\varphi({\bf a})=(n+1)!
\vol(i\Delta),\ 
\sum_{{\bf a}\in \Delta\cap (\mathbb{Z}/i)^n}i\varphi({\bf a})
{\bf a}=(n+1)!\int_{i\Delta}{\bf x}dv
\right\}.
\end{equation}
\end{prop}

Therefore we have the following by Proposition \ref{prop:3.4}.

\begin{prop}\label{prop:4.6}
Let $\Delta\subset \mathbb{R}^n$ be an $n$-dimensional integral Delzant
polytope. Then the Chow form
of $(X_\Delta,L_\Delta^i)$ is 
$T_{i\Delta}^{\mathbb{C}}$-semistable if and only if
\begin{equation}\label{e:4.2}
\frac{i^n(n+1)!\vol(\Delta)}{E_\Delta(i)}d_{i\Delta}\in
Ch_{i\Delta},
\end{equation}
where
$d_{i\Delta}\in
W(i\Delta)$ is
$$
d_{i\Delta}({\bf a})=1,\ \ {\bf a}\in \Delta\cap(\mathbb{Z}/i)^n.
$$
\end{prop}

\begin{proof}[Proof of Theorem \ref{thm:1.1}]

The condition \eqref{e:4.2} is equivalent to the following condition.

\begin{equation}\label{e:4.3}
\forall \varphi\in W(i\Delta),
\max\{(\varphi,\psi)\,|\,\psi\in Ch_{i\Delta}\}\ge
\frac{(n+1)!\vol(i\Delta)}{E_\Delta(i)}(\varphi,d_{i\Delta}).
\end{equation}

By Lemma $1.9$ of \cite{gkz} Chapter $7$ and
the definition of $g_\varphi$,
\eqref{e:4.3} is equivalent to
\begin{equation}\label{e:4.4}
\forall \varphi\in W(i\Delta),\int_\Delta g_\varphi dv\ge
\frac{\vol(\Delta)}{E_\Delta(i)}\sum_{{\bf a}\in \Delta\cap (\mathbb{Z}/i)^n}
g_\varphi({\bf a}).
\end{equation}
\end{proof}

Applying \eqref{e:1.1} to plus and minus of coordinate functions,
we have the following corollary.

\begin{cor}[\cite{ono}]\label{cor:4.7}
Let $\Delta\subset \mathbb{R}^n$ be an $n$-dimensional integral Delzant
polytope. If $(X_\Delta,L_\Delta^i)$ is Chow semistable for a positive
integer $i$ then
we have
\begin{equation}\label{e:4.5}
\sum_{{\bf a}\in \Delta\cap (\mathbb{Z}/i)^n}{\bf a}
=\frac{E_\Delta(i)}{\vol(\Delta)}
\int_\Delta {\bf x}\, dv.
\end{equation}
\end{cor}

\section{K-semistability}

For polarized manifolds, the notion of K-stability is defined by
Donaldson in \cite{d}. In toric case, restricting test configurations
to toric degenerations, he also defined the following notion.

\begin{defn}[\cite{d}]\label{def:5.1}
Let $(X_\Delta,L_\Delta)$ be an $n$-dimensional polarized toric manifold.
If the inequality
\begin{equation}\label{e:5.1}
\frac{\vol(\partial \Delta)}{\vol(\Delta)}\int_\Delta hdv
-\int_{\partial \Delta}hd\sigma\le 0
\end{equation}
holds for any rational convex piecewise linear function $h:\Delta\to
\mathbb{R}$, $(X_\Delta,L_\Delta)$ is called K-semistable for
toric degenerations.
\end{defn}

Ross and Thomas, in \cite{rt}, showed that asymptotically Chow semistable
polarized manifold is K-semistable. In toric case, we can partly prove
this fact by Theorem \ref{thm:1.1}
without any knowledge of Riemann-Roch theorem and test configurations
as follows.

\begin{prop}\label{prop:5.2}
If $(X_\Delta,L_\Delta)$ is asymptotically Chow semistable,
then $(X_\Delta,L_\Delta)$ is K-semistable for toric degenerations.
\end{prop}

\begin{proof}
Let $h:\Delta\to \mathbb{R}$ be a rational convex piecewise linear function.
Then $g=-h$ is a rational concave piecewise linear function on $\Delta$.
Thus there is a positive integer $i$ such that
$g\in PL_{\mathbb{Q}}(\Delta;i)$.
Since $(X_\Delta,L_\Delta)$ is asymptotically Chow semistable,
\begin{equation}\label{e:5.2}
P_\Delta(ki;g)=E_\Delta(ki)\int_{\Delta} g dv-
\vol(\Delta)\sum_{{\bf a}\in \Delta\cap (\mathbb{Z}/(ki))^n}
g({\bf a})\ge 0
\end{equation}
holds for any positive integer $k$ by Theorem \ref{thm:1.1}.
Note here that
\begin{equation}\label{e:5.3}
E_\Delta(ki)=(ki)^n\vol (\Delta)+\frac{(ki)^{n-1}}{2}\vol(\partial \Delta)+
O((ki)^{n-2}).
\end{equation}
Moreover by Lemma $3.3$ of \cite{zz},
\begin{equation}\label{e:5.4}
\sum_{{\bf a}\in \Delta\cap (\mathbb{Z}/ik)^n}g({\bf a})=
(ki)^n\int_\Delta gdv+\frac{(ki)^{n-1}}{2}\int_{\partial \Delta}gd\sigma+
O((ki)^{n-2})
\end{equation}
holds. Therefore we have

\begin{equation}\label{e:5.6}
\begin{split}
P_\Delta(ki;g)
=\frac{(ki)^{n-1}\vol(\Delta)}{2}\left(
\frac{\vol(\partial \Delta)}{\vol(\Delta)}\int_\Delta gdv-
\int_{\partial \Delta}gd\sigma
\right)+O((ki)^{n-2})
\end{split}
\end{equation}
holds. This proposition follows immediately from \eqref{e:5.2} and
\eqref{e:5.6}.
\end{proof}

By \eqref{e:5.6},
the following proposition trivially holds.

\begin{prop}\label{thm:1.2}
An $n$-dimensional polarized toric manifold $(X_\Delta,L_\Delta)$
is K-semistable for toric degenerations
if and only if the one of the following holds
for any positive integer $i$ and for any $g\in PL_{\mathbb{Q}}(\Delta,i)$:
\begin{equation}\label{e:1.2}
\begin{split}
& P_\Delta(ki;g)
> 0,\ \ \forall k\gg 0\ 
\text{ or }\ 
\frac{P_\Delta(ki;g)
}{(ki)^{n-1}}
\to 0\ \ (k\to \infty).
\end{split}
\end{equation}
\end{prop}

By Theorem \ref{thm:1.1},
asymptotic Chow semistablity of $(X_\Delta,L_\Delta)$ correoponds to 
non-negativity of the polynomial $P_\Delta(ki;g)$.
On the other hand, $(X_\Delta,L_\Delta)$ is K-semistable
for toric degenerations if and only if the leading coefficients
of the polynomial $P_\Delta(ki;g)$ is non-negative.
Hence asymptotic Chow semistability is much
stronger than
K-semistability. In fact we showed in \cite{osy} that
there exists a $7$-dimensional K-polystable toric Fano manifold
which is not asymptotically Chow semistable.
In this case, by Corollary $1.7$ of \cite{ono}, there is a rational
linear function
$l$ on $\Delta$ such that
$
P_\Delta(i;l)
<0
$
for any sufficiently large integer $i$.
Hence $l$ is a destabilizing object for asymptotic
Chow semistability.
On the other hand, since the Futaki invariant of this Fano manifold
vanishes, we see that the leading coefficient
of $P_\Delta(i;l)$
$$
\frac{\vol(\partial \Delta)}{\vol(\Delta)}\int_\Delta ldv
-\int_{\partial \Delta}ld\sigma
$$
vanishes.
Thus
$$
\frac{P_\Delta(i;l)}{i^{n-1}}\to 0\ \ (i\to \infty).
$$
Therefore $l$ is not a destabilizing object for K-semistability.

\section{Relative K-semistability}

The notion of relative K-stability was defined in \cite{s}.
As Definition \ref{def:5.1}, we defined the following.

\begin{defn}[\cite{zz}]\label{def:6.1}
Let $(X_\Delta,L_\Delta)$ be an $n$-dimensional polarized toric manifold.
If the inequality
\begin{equation}\label{e:6.1}
\int_\Delta \left(\frac{\vol(\partial \Delta)}{\vol(\Delta)}+
\theta_\Delta\right)hdv
-\int_{\partial \Delta}hd\sigma\le 0
\end{equation}
holds for any rational convex piecewise linear function $h:\Delta\to
\mathbb{R}$, $(X_\Delta,L_\Delta)$ is called relative K-semistable for
toric degenerations. Here
$\theta_\Delta:\Delta\to \mathbb{R}$
is the affine linear function, defined in Lemma $1.1$ of \cite{zz},
corresponding to the extremal vector field of 
$(X_\Delta,L_\Delta)$.
\end{defn}

For any sufficiently large integer $i$ and
a continuous function $g$ on $\Delta$ we have
\begin{equation}\label{e:6.2}
\sum_{{\bf a}\in \Delta\cap (\mathbb{Z}/i)^n}\theta_\Delta({\bf a})
g({\bf a})
=i^n \int_\Delta \theta_\Delta gdv
+O(i^{n-1}).
\end{equation}

Thus for any
$g\in PL_{\mathbb{Q}}(\Delta;i)$ and sufficiently large integer $k$
\begin{equation}\label{e:6.3}
\begin{split}
&Q_\Delta(ki;g):=E_\Delta(ki)\int_{\Delta} g dv+
\vol(\Delta)\sum_{{\bf a}\in \Delta\cap (\mathbb{Z}/(ki))^n}
(\frac{1}{2ki}\theta_\Delta({\bf a})-1)g({\bf a})\\
&=\frac{(ki)^{n-1}\vol(\Delta)}{2}\left\{
\int_\Delta \left(\frac{\vol(\partial \Delta)}{\vol(\Delta)}
+\theta_\Delta\right)gdv-
\int_{\partial \Delta}gd\sigma
\right\}+O((ki)^{n-2}).
\end{split}
\end{equation}

Thus, as Proposition \ref{thm:1.2}, we see the following.

\begin{prop}\label{thm:1.3}
An $n$-dimensional polarized toric manifold
$(X_\Delta,L_\Delta)$ is relative K-semistable for
toric degenerations
if and only if the one of the following holds
for any positive integer $i$ and for any $g\in PL_{\mathbb{Q}}(\Delta,i)$:
\begin{equation}\label{e:1.3}
\begin{split}
& Q_\Delta(ki;g)
\ge 0,\ \forall k\gg 0\ 
\text{ or }\ 
\frac{Q_\Delta(ki;g)}{(ki)^{n-1}}\to 0\ \ (k\to \infty).
\end{split}
\end{equation}
\end{prop}

Finally we see that some GIT-semistability implies
relative K-semistability for toric degenerations.

\begin{defn}\label{def:6.2}
$(X_\Delta,L_\Delta^i)$ is relative Chow semistable for
$T^\mathbb{C}_{i\Delta}$-action if
\begin{equation}\label{e:6.4}
\frac{i^n(n+1)!\vol(\Delta)}{E_\Delta(i)}(d_{i\Delta}-\frac{1}{2i}\theta_
\Delta)\in Ch_{i\Delta}.
\end{equation}
$(X_\Delta,L_\Delta)$ is asymptotically relative Chow semistable
in toric sense if there exists a positive integer $i_0$ such that
for any integer $i\ge i_0$ \eqref{e:6.4} holds.
\end{defn}

\begin{prop}\label{prop:6.3}
If $(X_\Delta,L_\Delta)$ is asymptotically relative Chow semistable
in toric sense, then $(X_\Delta,L_\Delta)$ is relative
K-semistable for toric degenerations. 
\end{prop}

\begin{proof}
Suppose that
$(X_\Delta,L_\Delta)$ is asymptotically relative Chow semistable
in toric sense. It is equivalent to
\begin{equation}\label{e:6.5}
\forall \varphi\in W(i\Delta),\int_{\Delta} g_\varphi dv\ge
\frac{\vol(\Delta)}{E_\Delta(i)}\sum_{{\bf a}\in \Delta\cap (\mathbb{Z}/i)^n}
(1-\frac{1}{2i}\theta_\Delta({\bf a}))\varphi({\bf a})
\end{equation}
for any integer $i\ge i_0$. Hence it implies that
for any $g\in PL(\Delta;i)$
\begin{equation}\label{e:6.6}
\int _\Delta gdv\ge \frac{\vol(\Delta)}{E_\Delta(i)}\sum_
{{\bf a}\in \Delta\cap (\mathbb{Z}/i)^n}(1-\frac{\theta_\Delta({\bf a})}
{2i})g({\bf a}).
\end{equation}
\end{proof}


\begin{thebibliography}{99}

\bibitem{d}S.~K.~Donaldson, Scalar curvature and stability of toric
varieties, J. Differential Geom. 62 (2002), 289--349.


\bibitem{ful}W.~Fulton, Introduction to Toric Varieties, Number 131, in
Annals of Mathematics Studies, Princeton University Press, Princeton,
NJ, 1993.



\bibitem{gkz}I.~M.~Gelfand, M.~M.~Kapranov and A.~V.~Zelevinsky,
Discriminants, Resultants and Multidimensional Determinants, 
Mathematics: Theory $\&$ Applications. Birkh\"auser Boston Inc., Boston,
MA, 1994.

\bibitem{ksz}M.~M.~Kapranov, B.~Sturmfels and A.~V.~Zelevinsky,
Chow polytopes and general resultants,
Duke Math. J. 67 (1992), no. 1, 189--218.


\bibitem{m}D.~Mumford, Stability of projective varieties,
Enseignement Math. (2), 23 (1977), no. 1--2, 39--110.

\bibitem{git}D.~Mumford, J.~Fogarty and F.~Kirwan, Geometric Invariant
Theory, Third edition, Ergebnisse der Mathematik und ihrer Grenzgebiete (2),
34, Springer-Verlag, Berlin, 1994.

\bibitem{o}T.~Oda, Convex bodies and algebraic geometry,
Ergebnisse der Mathematik und ihrer Grenzgebiete (3),
15, Springer-Verlag, Berlin, 1988.

\bibitem{ono}H.~Ono, A necessary condition for Chow semistability of
polarized toric manifolds, arXiv:1003.1553.

\bibitem{osy}H.~Ono, Y.~Sano and N.~Yotsutani, An example of asymptotically
Chow unstable manifolds with constant scalar curvature, arXiv:0906.3836.


\bibitem{rt}J.~Ross and R.~Thomas, A study of the Hilbert-Mumford criterion
for the stability of projective varieties. J. Algebraic Geom. 10
(2007), 201--255.

\bibitem{s}G.~Sz\'ekelyhidi, Extremal metrics and K-stability,
Bull. Lond. Math. Soc. 39 (2007), no. 1, 76--84.

\bibitem{zz}B.~Zhou and X.~Zhu, Relative K-stability and modified K-energy
on toric manifolds, Adv. Math. 219 (2008), 1327--1362.



\end{thebibliography}
\end{document}